\documentclass[leqno]{amsart}
\usepackage{amssymb}
\usepackage{mathrsfs}
\usepackage{amsmath, amsfonts, vmargin, enumerate}
\usepackage{graphicx}
\usepackage{color}
\usepackage{verbatim}
\usepackage{amsthm}
\usepackage{latexsym, bm}

\usepackage{euscript}
\usepackage{dsfont}

\input xypic

\xyoption {all}

 \makeatletter

\newtheorem{thm}{Theorem}

\newtheorem{lem}[thm]{Lemma}
\newtheorem{defi}[thm]{Definition}
\newtheorem{remark}[thm]{Remark}
\newtheorem{example}[thm]{Example}
\newtheorem{pb}[thm]{Problem}

\newenvironment{rk}{\begin{remark}\rm}{\end{remark}}

\newcommand{\real}{{\mathbb R}}
\newcommand{\nat}{{\mathbb N}}
\newcommand{\ent}{{\mathbb Z}}
\newcommand{\com}{{\mathbb C}}

\newcommand{\T}{{\mathbb T}}

\newcommand{\A}{{\mathcal A}}
\newcommand{\B}{{\mathcal B}}

\renewcommand{\a}{\alpha}
\renewcommand{\b}{\beta}
\newcommand{\g}{\gamma}
\newcommand{\Ga}{\Gamma}

\renewcommand{\d}{\delta}

\newcommand{\e}{\varepsilon}
\newcommand{\f}{\varphi}

\renewcommand{\l}{\lambda}
\renewcommand{\O}{\Omega}
\renewcommand{\o}{\omega}
\newcommand{\s}{\sigma}

\newcommand{\ot}{\otimes}

\newcommand{\8}{\infty}

\newcommand{\la}{\langle}
\newcommand{\ra}{\rangle}
\newcommand{\wt}{\widetilde}

\newcommand{\n}{\noindent}

\newcommand{\les}{\lesssim}

\newcommand{\be}{\begin{align*}}
\newcommand{\ee}{\end{align*}}
\newcommand{\beq}{\begin{equation}}
\newcommand{\eeq}{\end{equation}}
\newcommand{\beqn}{\begin{equation*}}
\newcommand{\eeqn}{\end{equation*}}

\begin{document}

\title[Vector-valued Littewood-Paley-Stein theory]{Vector-valued Littewood-Paley-Stein theory for semigroups II}


\author[Quanhua  Xu]{Quanhua Xu}
\address{Institute for Advanced Study in Mathematics, Harbin Institute of Technology,  Harbin 150001, China; and Laboratoire de Math{\'e}matiques, Universit{\'e} de Bourgogne Franche-Comt{\'e}, 25030 Besan\c{c}on Cedex, France; and Institut Universitaire de France}
\email{qxu@univ-fcomte.fr}

\date{}
\maketitle

 \begin{abstract}
 Inspired by a recent work of Hyt\"onen and Naor, we solve a problem left open in our previous work joint with Mart\'{\i}nez and Torrea  on the vector-valued Littlewood-Paley-Stein theory for symmetric diffusion semigroups. We prove a similar result in the discrete case, namely, for any $T$ which is the square of a symmetric Markovian operator on a measure space  $(\Omega, \mu)$. Moreover, we show that  $T\ot{\rm Id}_X$ extends to an analytic contraction on $L_p(\Omega; X)$ for any $1<p<\infty$ and any uniformly convex Banach space $X$.
   \end{abstract}

\bigskip


\section{Introduction}


Let $(\O, \A, \mu)$ be  a $\s$-finite measure space.  By a \emph{symmetric diffusion semigroup} on $(\O,\A, \mu)$ in Stein's sense \cite[section~III.1]{stein}, we mean a family  $\{T_t\}_{t>0}$ of linear maps satisfying the following properties:
 \begin{enumerate}[$\bullet$]
 \item $T_t$ is a contraction on $L_p(\O)$ for every $1\le p\le \8$;
  \item $T_tT_s=T_{t+s}$;
 \item $\lim_{t\to 0}T_tf=f$ in $L_2(\O)$ for every $f\in L_2(\O)$;
 \item $T_t$ is positive (i.e. positivity preserving) and $T_t1=1$;
 \item $T_t$ is  selfadjoint on $L_2(\O)$.
\end{enumerate}
It is a classical fact that the orthogonal projection from $L_2(\O)$ onto the  fixed point subspace of
$\{T_t\}_{t> 0}$ extends to a contractive projection on $L_p(\Omega)$ for
every $1\le p\le\infty$. We will denote this projection by $\mathsf F$. Then $\mathsf F$ is also positive and   $\mathsf F\big(L_p(\Omega)\big)$ is
 the fixed point subspace of $\{T_t\}_{t> 0}$ on $L_p(\Omega)$ (cf. e.g. \cite{ds}).

Stein  proved in \cite[chapter~IV]{stein} the following result which considerably extends the classical  inequality on the Littlewood-Paley  $g$-function in harmonic analysis: For every $1<p<\8$
 \beq\label{LPS}
 \|f-\mathsf F(f)\|_{L_p(\O)}\approx \left\|\left(\int_0^\8\Big |t\frac{\partial}{\partial t} T_t f\Big|^2\,\frac{dt}t\right)^{1/2}\right\|_{L_p(\O)}\,,\quad\forall\, f\in L_p(\O),
 \eeq
where the equivalence constants depend only on $p$.

The vector-valued Littlewood-Paley-Stein theory was developed in \cite{X, MTX}. Given a Banach space $X$, we denote by $L_p(\O; X)$ the usual $L_p$ space of strongly measurable functions from $\O$ to $X$. It is a well known elementary fact that if $T$ is a positive bounded operator on $L_p(\O)$ with $1\le p\le\8$, then $T\ot{\rm Id}_X$ is bounded on $L_p(\O; X)$ with the same norm. For notational convenience, throughout this paper, we will denote $T\ot{\rm Id}_X$ by $T$ too. Thus $\{T_t\}_{t>0}$ is also a semigroup of contractions on $L_p(\O; X)$ for any Banach space $X$.

The one-sided vector-valued extension of \eqref{LPS} was obtained in \cite{MTX} not for the semigroup $\{T_t\}_{t> 0}$ itself but for its  subordinated Poisson semigroup $\{P_t\}_{t>0}$ that is
defined by
 $$P_tf=\frac{1}{\sqrt\pi}\,\int_0^\infty \frac{e^{-s}}{\sqrt s}\,T_{\frac{t^2}{4s}}\,f\,ds.$$
$\{P_t\}_{t> 0}$ is again a symmetric diffusion semigroup. Recall that if $A$ denotes the negative infinitesimal generator of $\{T_t\}_{t>0}$, then $P_t=e^{-\sqrt A\,t}$.

Let $1<q<\8$. Recall that a Banach space $X$ is of {\it martingale cotype} $q$ if there exists a positive constant $C$ such that every finite $X$-valued $L_q$-martingale $(f_n)$ defined on some probability space  satisfies the following inequality
 $$\sum_n\mathbb{E}\big\|f_n-f_{n-1}\big\|_X^q\le C^q\sup_n\mathbb{E}\big\|f_n\big\|_X^q\,,$$
where $\mathbb{E}$ denotes the expectation on the underlying probability space.  We then must have $q\ge2$. $X$ is of {\it martingale type} $q$ if the reverse inequality holds. It is easy to see that $X$ is of martingale cotype $q$ iff the dual space $X^*$ is of martingale type $q'$, where $q'$ denotes the conjugate index of $q$. We refer to \cite{pis1,pis3} for more information.

\smallskip

The following is the principal result of \cite{MTX}.  In the sequel, we will use the abbreviation  $\partial=\partial/\partial t$.

\begin{thm}[Mart\'{\i}nez-Torrea-Xu]\label{MTX}
 Let $1< q<\8$ and $X$ be a Banach space.
 \begin{enumerate}[\rm(i)]
 \item $X$ is of martingale cotype $q$ iff for every $1<p<\8$ $($or equivalently, for some $1<p<\8)$ there exists a constant $C$ such that every  subordinated Poisson semigroup $\{P_t\}_{t>0}$ as above  satisfies the following inequality
  $$\left\|\left(\int_0^\8\big\|t\, \partial P_t f\big\|_X^q\,\frac{dt}t\right)^{1/q}\right\|_{L_p(\O)}\le C\, \big\|f\big\|_{L_p(\O; X)}\,,\quad\forall\, f\in L_p(\O; X).$$
 \item $X$ is of martingale type $q$ iff for for every $1<p<\8$ $($or equivalently, for some $1<p<\8)$ there exists a constant $C$ such that every  subordinated Poisson semigroup $\{P_t\}_{t>0}$ as above  satisfies the following inequality
  $$\big\|f\big\|_{L_p(\O; X)}\le \big\|\mathsf F(f)\big\|_{L_p(\O; X)}+C \left\|\left(\int_0^\8\big\|t\, \partial P_t f\big\|_X^q\,\frac{dt}t\right)^{1/q}\right\|_{L_p(\O)}\,,\quad\forall\, f\in L_p(\O; X).$$
  \end{enumerate}
\end{thm}

Note that  the above theorem for the Poisson semigroup of the torus $\T$ was first proved in \cite{X}. The main problem left open in \cite{MTX} asks whether the  theorem holds for the semigroup $\{T_t\}_{t> 0}$ itself  instead of its subordinated Poisson semigroup $\{P_t\}_{t>0}$ (see Problem~2 on page~447 of \cite{MTX}). Very recently, Hyt\"onen and Naor \cite{HN} proved that the answer is affirmative for the heat semigroup of $\real^n$ and for $p=q$; the resulting inequality plays a key role in their work  on  the approximation of Lipschitz functions by affine maps.  Stimulated by their result and using a clever idea of them, we are able to resolve the problem in  full generality.

\begin{thm}\label{heat}
  Let $X$ be a Banach space and $k$ a positive integer.
 \begin{enumerate}[\rm(i)]
 \item If $X$ is of martingale cotype $q$ with $2\le q<\8$, then for every  symmetric diffusion semigroup $\{T_t\}_{t>0}$ and for every $1<p<\8$ we have
  $$\left\|\left(\int_0^\8\big\|t^k \partial^k T_t f\big\|_X^q\,\frac{dt}t\right)^{1/q}\right\|_{L_p(\O)}\le C\, \big\|f\big\|_{L_p(\O; X)}\,,\quad\forall\, f\in L_p(\O; X),$$
 where $C$ is a constant depending only on $p, q, k$ and the martingale cotype $q$ constant of $X$.
 \item If $X$ is of martingale type $q$ with $1< q\le2$, then for every  symmetric diffusion semigroup $\{T_t\}_{t>0}$  and for every $1<p<\8$ we have
  $$\big\|f\big\|_{L_p(\O; X)}\le \big\|\mathsf F (f)\big\|_{L_p(\O; X)}+C \left\|\left(\int_0^\8\big\|t^k \partial^k T_t f\big\|_X^q\,\frac{dt}t\right)^{1/q}\right\|_{L_p(\O)}\,,\quad\forall\, f\in L_p(\O; X),$$
  where $C$ is a constant depending only on $p, q, k$ and the martingale type $q$ constant of $X$.
  \end{enumerate}
\end{thm}

\begin{rk}
Applied to the heat semigroup $\{H_t\}_{t>0}$ of $\real^n$, the above theorem implies a dimension free estimate for the $g$-function associated to $\{H_t\}_{t>0}$:
  $$\left\|\left(\int_0^\8\big\|t \partial H_t f\big\|_X^q\,\frac{dt}t\right)^{1/q}\right\|_{L_p(\real^n)}\le C\, \big\|f\big\|_{L_p(\real^n; X)}\,,\quad\forall\, f\in L_p(\real^n; X)$$
when $X$ is of martingale cotype $q$. Compare this with \cite[Theorem~17]{HN} (and the paragraph thereafter).
 \end{rk}

 \begin{rk}
Theorem~\ref{heat} allows one to improve some recent results of Hong and Ma on vector-valued variational inequalities associated to symmetric diffusion semigroups. For instance, using it, one can extend \cite[Theorem~5.2]{hm2} to any  Banach space $X$  of martingale cotype $q_0$. See also \cite{hm1} for related results in the Banach lattice case.
 \end{rk}

Theorem~\ref{heat}  admits a discrete analogue.  First recall that a power bounded operator $R$ on a Banach space $Y$ is said to be {\it analytic} if
 $$\sup_{n\ge1}n\big\|R^n(R-1)\big\|<\8,$$
where the norm is the operator norm on $Y$. It is known that the analyticity of $R$ is equivalent to
 $$
 \sup_{z\in\com, |z|>1}|1-z|\,\big\|(z-R)^{-1}\big\|<\8.
 $$
Moreover, if $R$ is analytic, its spectrum $\s(R)$ is contained in  $\overline{B_{\g}}$ for some $0<\g<\pi/2$, where $B_{\g}$  denotes  the Stolz domain  which is the interior of the convex hull of $1$ and the disc $D(0,\,\sin\g)$ (see figure~1). We refer to \cite{b, N} for more information.

\begin{figure}[ht]
\vspace*{2ex}
\begin{center}
\includegraphics[scale=0.4]{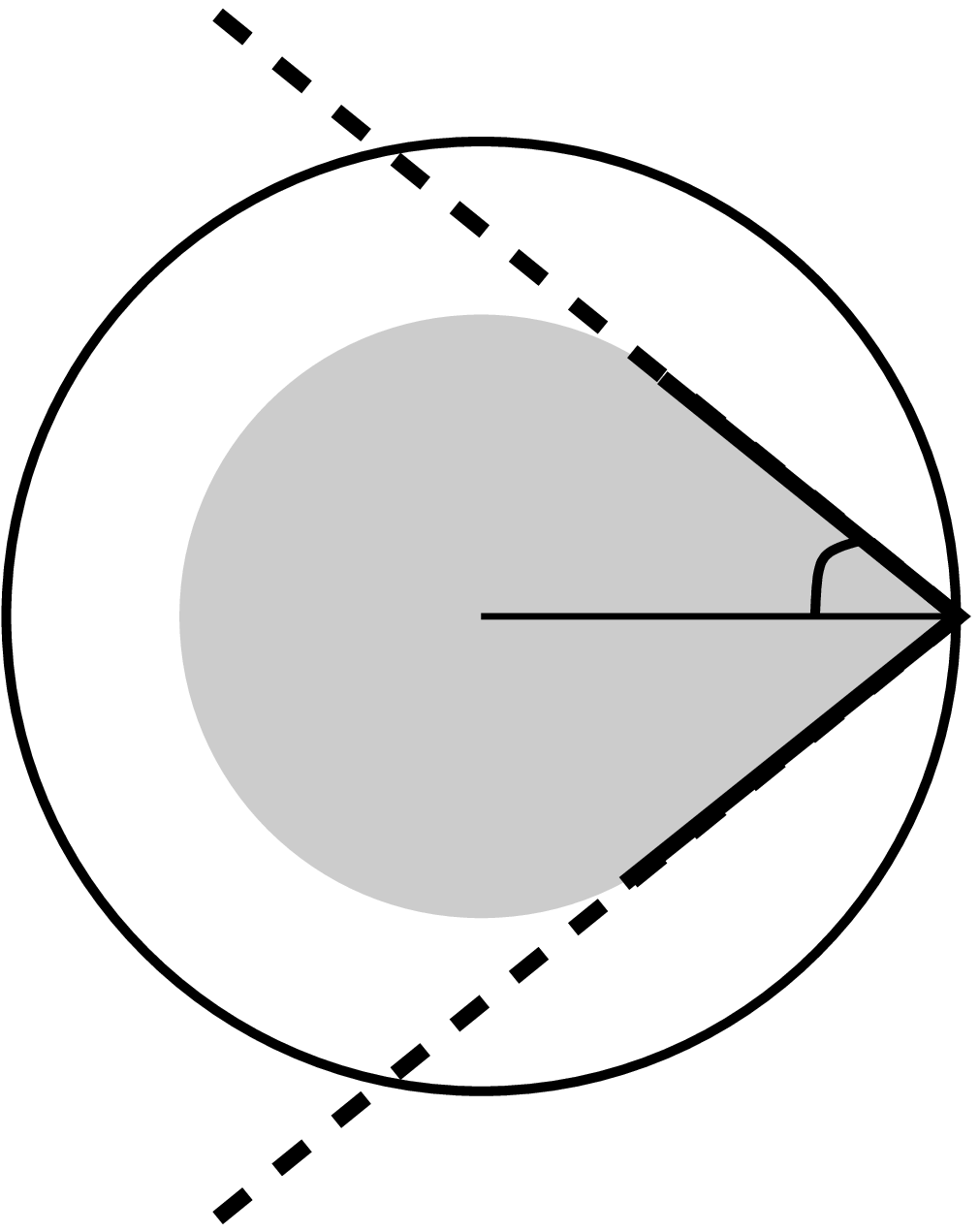}
\begin{picture}(0,0)
\put(-2,65){{\footnotesize $1$}}
\put(-68,65){{\footnotesize $0$}}
\put(-32,81){{\footnotesize $\gamma$}}
\put(-55,85){{\small $B_\gamma$}}
\end{picture}
\end{center}
\caption{\label{f1}}
 \end{figure}

Now consider  a \emph{symmetric Markovian operator} $T$ on $(\O,\A,\mu)$, that is, $T$ satisfies the following conditions:
\begin{enumerate}[$\bullet$]
 \item $T$ is a linear contraction on $L_p(\O)$ for every $1\le p\le \8$;
 \item $T$ is positivity preserving and $T1=1$;
 \item $T$ is  a selfadjoint operator on $L_2(\O)$.
\end{enumerate}
With a slight abuse of notation, we use again $\mathsf F$  to denote the projection on the fixed point subspace of $T$. Both $T$ and $\mathsf F$ extend to contractions on $L_p(\O; X)$ for any Banach space $X$.  In the following two theorems, $T=S^2$ with $S$ a symmetric Markovian operator, so $T$ is a symmetric Markovian operator too, The following is the discrete analogue of a theorem of Pisier \cite{pis2} for semigroups.

 \begin{thm}\label{analyticity-discrete}
  Let $T=S^2$ with $S$ a symmetric Markovian operator, $1<p<\8$ and $X$ be a uniformly convex Banach space. Then the extension of $T$ to $L_p(\O; X)$  is analytic. More precisely, there exist constants $C$ and $\g\in(0,\,\pi/2)$, depending only on $p$ and the modulus of uniform convexity of $X$, such that
  \beq\label{resolvent}
  \s(T)\subset \overline{B_{\g}} \;\text{ and }\; \big\|(z-T)^{-1}\big\|\le \frac{C}{|1-z|}\,,\quad\forall z\in\com\setminus \overline{B_{\g}}.
  \eeq
  \end{thm}

The discrete analogue of Theorem~\ref{heat}  is the following

 \begin{thm}\label{discrete}
  Let $T=S^2$ be as above and $1<p<\8$.
 \begin{enumerate}[\rm(i)]
 \item If $X$ is of martingale cotype $q$ with $2\le q<\8$, then
  $$\left\|\left(\sum_{n=1}^\8 n^{q-1}\big\|(T^n -T^{n-1}) f\big\|_X^q \right)^{1/q}\right\|_{L_p(\O)}\le C\, \big\|f\big\|_{L_p(\O; X)}\,,\quad\forall\, f\in L_p(\O; X),$$
where the constant $C$ depends only on $p,q$ and the martingale cotpye $q$ constant of $X$.
 \item If $X$ is of martingale type $q$ with $1< q\le2$, then
  $$\big\|f\big\|_{L_p(\O; X)}\le C \left\|\left(\big\|\mathsf Ff\big\|_X^q+\sum_{n=1}^\8n^{q-1}\big\|(T^n -T^{n-1}) f\big\|_X^q \right)^{1/q}\right\|_{L_p(\O)}\,,\quad\forall\, f\in L_p(\O; X),$$
  where the constant $C$ depends only on $p,q$ and the martingale tpye $q$ constant of $X$.
   \end{enumerate}
\end{thm}

\begin{rk}
  If the inequality in Theorem~\ref{discrete} (i)  holds for every positive symmetric Markovian operator $T$,  then the corresponding inequality of Theorem~\ref{MTX} holds for every  subordinated Poisson semigroup $\{P_t\}_{t>0}$. Thus $X$ is of martingale cotype $q$.  Therefore, the validity of the inequality in Theorem~\ref{discrete} (i)  characterizes the martingale cotype $q$ of $X$. A similar remark applies to  part (ii).
\end{rk}

\begin{rk}
  It is worth to note that all constants involved in the preceding  theorems are independent of  the semigroup $\{T_t\}_{t>0}$  or contraction $T$ in consideration. They depend only on the indices $p, q$ and the relevant geometric constants of the space $X$.
\end{rk}

The preceding three theorems will be proved in the next three sections. The proofs of Theorem~\ref{heat} and Theorem~\ref{discrete} follow the same pattern although the latter one  is  more involved. The last section contains some open problems.

\smallskip

We will use the symbol $\les$ to denote an inequality up to a constant factor; all constants will depend only on $X$, $p, q$,  etc. but never on the function $f$ in consideration.


\section{A spectral estimate}


This section contains a spectral estimate for positive symmetric Markovian operators. Let $X$ be a uniformly convex Banach space and $1<p<\8$. Then $Y=L_p(\O; X)$ is uniformly convex too. By Pisier's renorming theorem \cite{pis1}, we can assume that $Y$ is  uniformly convex of power type $q$ for some $2\le q<\8$, namely,
 \beq\label{conx}
 \left\|\frac{x+y}2\right\|^q+\d\,\left\|\frac{x-y}2\right\|^q\le \frac12\big(\|x\|^q+\|y\|^q\big),\quad \forall\, x, y\in Y
 \eeq
for some positive constant $\d$. Note that the above inequality implies the martingale cotype $q$ of $Y$. Conversely, if $Y$ is of martingale cotype $q$, then it admits an equivalent norm which satisfies \eqref{conx}. Let $T=S^2$ with $S$ a symmetric Markovian operator on $(\O, \A, \mu)$. We extend $T$ to a contraction on $Y$, still denoted by $T$. In the following the norm and spectrum of $T$ is taken for $T$ viewed as an operator on $Y$.

\begin{lem}\label{spectral estimate}
 Under the above assumptions we have
 \begin{enumerate}[\rm(i)]
 \item $\displaystyle\|1-T\|\le \min\big(\frac32,\, 2(1-\frac{\d}{2^q})^{1/q}\big)<2$;
 \item the spectrum of $T$ is contained in a Stolz domain $\overline{B_{\g}}$ for some $\g\in(0,\, \pi/2)$ depending only on $\d$ and $q$ in \eqref{conx}.
 \end{enumerate}
  \end{lem}

Part (i) above is already contained in \cite{pis2} (see, in particular, Remark~1.8 there). In fact, our proof below is modeled on that of  \cite[Lemma~1.5]{pis2}.  As in \cite{pis2}, We will need the following one step version of Rota's dilation theorem for positive symmetric Markovian operators. We refer to \cite[Chapter~IV]{stein} for its proof as well as its full version.

\begin{lem}[Rota]\label{Rota}
 Let $T=S^2$ with $S$ a  symmetric Markovian operator on $(\O, \A, \mu)$. Then there exist a larger measure space $(\wt{\O}, \wt{\A},\wt{\mu})$ containing $(\O, \A, \mu)$, and a $\s$-subalgebra $\B$ of $\wt{\A}$ such that
 $$Tf=\mathbb{E}_\A \mathbb{E}_\B f,\quad\forall f\in L_p(\O, \A, \mu),$$
where $\mathbb{E}_\A$ denotes the conditional expectation relative to $\A$ (and similarly for $\mathbb{E}_\B$).
\end{lem}

\begin{proof}[Proof of Lemma~\ref{spectral estimate}]
  Rota's dilation  extends to $X$-valued functions:
 $$T=\mathbb{E}_\A \mathbb{E}_\B\big|_Y.$$
Here we have used our usual convention that $\mathbb{E}_\A\ot{\rm Id}_X$ and $\mathbb{E}_\B\ot{\rm Id}_X$ are abbreviated to $\mathbb{E}_\A$ and $\mathbb{E}_\B$, respectively. Thus for any $\l\in\com$ (with $P=\mathbb{E}_\B$)
 $$\l +T=\mathbb{E}_\A(\l +P)\big|_Y.$$
Let $y$ be a unit vector in $Y$. Using \eqref{conx}, we get
 $$\left\|\frac{\l y+Py}2\right\|^q+\d\,\left\|\frac{\l y-Py}2\right\|^q\le \frac12\,(|\l|^q+1).$$
However (noting that $P$ is a contractive projection),
 $$\big\|\l y-Py\big\|\ge |1-\l| \,\big\|Py\big\|\ge |1-\l| \big(\big\|\l y+Py\big\|-|\l|\big)\ge  |1-\l| \big(\big\|\l y+Ty\big\|-|\l|\big) .$$
When $\|\l y+Ty\|$ approaches  $\|\l+T\|$, we then deduce
 \beq\label{spectral}
 \left\|\frac{\l +T}2\right\|^q+\d\, |1-\l|^q\left(\frac{\big\|\l +T\big\|-|\l|}2\right)^q\le \frac12\,(|\l|^q+1).
 \eeq
In particular, for $\l=-1$ we obtain
  $$\|1-T\|^q+\d\, 2^q(\|1-T\|-1)^q\le 2^q,$$
which implies
 $$\|1-T\|\le\min\big(\frac32,\, 2(1-\frac{\d}{2^q})^{1/q}\big).$$
This is part (i). On the other hand, if $\l\in\s(T)$, then \eqref{spectral} yields
  $$|\l|^q+\d \,|1-\l|^q\, |\l|^q\le \frac12\big(|\l|^q+1\big),$$
 whence
  $$|1-\l|\,|\l|\le \left(\frac{q}{2\d}\right)^{1/q}\, (1-|\l|).$$
 The last inequality implies  (in fact, is equivalent to) that $\l\in \overline{B_{\g}}$ for some $\g\in (0,, \pi/2)$ depending only on the constant $\big(q/(2\d)\big)^{1/q}$. The proof of the lemma is thus complete.
\end{proof}

Lemma~\ref{spectral estimate} (i) implies the following result which is \cite[Remark~1.8]{pis2}.

\begin{lem}\label{analyticity0}
Let $X$ and $p$ be as above and $\{T_t\}_{t>0}$ be a symmetric diffusion semigroup  on $(\O, \A, \mu)$. Then the extension of $\{T_t\}_{t>0}$ to $Y=L_p(\O; X)$ is analytic. Consequently, $\{t\partial T_t\}_{t>0}$ is a uniformly bounded family of operators on $Y$, namely,
 \beq\label{pisier}
 \sup_{t>0}\big\|t\partial T_t\big\|\le C,
 \eeq
where $C$ is a constant depending only on $\d$ and $q$ in \eqref{conx}.
\end{lem}

\begin{proof}
 Applying Lemma~\ref{spectral estimate} to $T=T_t$, we get
 $$\sup_{t>0}\big\|1-T_t\big\|\le \min\big(\frac32,\, 2(1-\frac{\d}{2^q})^{1/q}\big)<2.$$
Then using Kato's characterization of analytic semigroups in \cite{K}, we deduce \eqref{pisier}.
\end{proof}


\section{Proof of Theorem~\ref{heat}}
\label{pf-heat}


This section is devoted to the proof of Theorem~\ref{heat}. Let us first note that assertion (ii) follows easily from (i) by duality. Indeed, let  $\{e_\l\}$ be the resolution of the identity of $\{T_t\}_{t>0}$ on $L_2(\O)$:
 $$T_tf=\int_0^\8 e^{-\l t}de_\l f,\quad f\in L_2(\O).$$
Then
 $$\partial^k T_tf=(-1)^k\int_0^\8 \l^k e^{-\l t}de_\l f.$$
It thus follows that
 \begin{align*}
 \int_\O\int_0^\8 \big|t^k\partial^k T_tf\big|^2\,\frac{dt}t\,d\mu
 &=\int_0^\8 \int_0^\8 t^{2k}\l^{2k} e^{-2\l t}d\la e_\l f,\,f\ra\,\frac{dt}t\\
 &=4^{-k}\int_0^\8 \int_0^\8 t^{2k} e^{-t}\,\frac{dt}t\,d\la e_\l f,\,f\ra\\
 &=4^{-k}(2k-1)! \int_0^\8 d\la e_\l f,\,f\ra\\
 &=4^{-k}(2k-1)! \int_\O |f-\mathsf F (f)|^2d\mu.
 \end{align*}
By polarization, for $f, g\in L_2(\O)$ we have
 $$\int_\O (f-\mathsf F (f))(g-\mathsf F (g))d\mu = \frac{4^k}{(2k-1)!}\,\int_\O\int_0^\8 \big(t^k\partial^k T_tf\big)\big(t^k\partial^k T_tg\big)\,\frac{dt}t\,d\mu.$$
We then deduce that for any $f\in L_1(\O)\cap L_\8(\O)\ot X$ and $g\in L_1(\O)\cap L_\8(\O)\ot X^*$
 $$\int_\O \la g-\mathsf F (g),f-\mathsf F (f)\ra d\mu = \frac{4^k}{(2k-1)!}\,\int_\O\int_0^\8 \la t^k\partial^k T_tg, t^k\partial^k T_tf\ra \,\frac{dt}t\,d\mu,$$
where $\la\,,\,\ra$ denotes the duality bracket between $X$ and  $X^*$. Hence
 \begin{align*}
 \left|\int_\O \la g-\mathsf F (g),f-\mathsf F (f)\ra d\mu\right|
 \le\frac{4^k}{(2k-1)!}&\left\|\left(\int_0^\8\big\|t^k \partial^k T_t g\big\|_{X^*}^{q'}\frac{dt}t\right)^{1/{q'}}\right\|_{L_{p'}(\O)}\\
 \cdot\,&\left\|\left(\int_0^\8\big\|t^k \partial^k T_t f\big\|_X^q\frac{dt}t\right)^{1/q}\right\|_{L_{p}(\O)},
  \end{align*}
where $r'$ is the conjugate index of $r$. Under the assumption of (ii)  and by duality, we have that  $X^*$ is of martingale cotype $q'$. Therefore,  (i) implies
 $$ \left\|\left(\int_0^\8\big\|t^k \partial^k T_t g\big\|_{X^*}^{q'}\,\frac{dt}t\right)^{1/q'}\right\|_{L_{p'}(\O)}\le \frac{4^k C}{(2k-1)!}\, \big\|g\big\|_{L_{p'}(\O; X^*)}\,.$$
Combining the previous inequalities and taking the supremum over all $g$ in the unit ball of $L_{p'}(\O; X^*)$, we derive assertion (ii).

\smallskip

Thus we are left to showing assertion (i). In the rest of this section, we will assume that  $X$ is a Banach space of martingale cotype $q$ with $2\le q<\8$.
The following lemma, due to Hyt\"onen and Naor \cite[Lemma 24]{HN}, will play an important role in our argument.
\begin{lem}[Hyt\"onen-Naor]\label{HN-lem}
 For any $f\in L_q(\O;X)$ we have
 $$\left(\int_0^\8\big\|(T_t-T_{3t})f\big\|_{L_q(\O;X)}^q\,\frac{dt}{t}\right)^{1/q}\les\big\|f\big\|_{L_q(\O;X)}\,,\quad\forall\, f\in L_q(\O;X)\,.$$
\end{lem}

Based on Rota's dilation theorem quoted in the previous section, the proof is simple. Below is the main idea. First write
  \begin{align*}
  \int_0^\8\big\|(T_t-T_{3t})f\big\|_{L_q(\O;X)}^q\,\frac{dt}{t}
 &=\sum_{k\in\ent}\int_{3^k}^{3^{k+1}}\big\|(T_t-T_{3t})f\big\|_{L_q(\O;X)}^q\,\frac{dt}{t}\\
 &=\int_1^3\sum_{k\in\ent}\big\|(T_{3^kt}-T_{3^{k+1}t})f\big\|_{L_q(\O;X)}^q\,\frac{dt}{t}\,.
  \end{align*}
Then  Rota's dilation theorem allows us to  turn $\{T_{3^kt}-T_{3^{k+1}t}\}_k$ for each fixed $t$ into a martingale difference sequence.

\smallskip

The following lemma shows  Theorem~\ref{heat} in the case of $p=q$.

\begin{lem}\label{p=q}
 Let $k$ be a positive integer. Then
 \beq\label{pqk}
 \left(\int_0^\8\big\|t^k\partial^k T_tf\big\|_{L_q(\O;X)}^q\,\frac{dt}{t}\right)^{1/q}\les\big\|f\big\|_{L_q(\O;X)}\,,\quad\forall\, f\in L_q(\O;X),
 \eeq
 where the relevant constant depends on $k$ and the martingale cotype $q$ constant of $X$.
 \end{lem}

\begin{proof}
We will use the idea of the proof of Theorem~17 of \cite{HN}. By virtue of  the identity $\partial T_{t+s}=\partial T_t\,T_s$, we write
 $$
 \partial T_tf
 =\sum_{k=-1}^\8\big(\partial T_{2^{k+1}t}- \partial T_{2^{k+2}t}\big)f
 =\sum_{k=-1}^\8\partial T_{2^{k}t}\big(T_{2^{k}t}- T_{3\cdot2^{k}t}\big)f.
 $$
Then by the triangle inequality we get
  \begin{align*}
  \left(\int_0^\8\big\|t\partial T_tf\big\|_{L_q(\O;X)}^q\,\frac{dt}{t}\right)^{1/q}
  &\le \sum_{k=-1}^\8 \left(\int_0^\8\big\|t\partial T_{2^{k}t}\big(T_{2^{k}t}- T_{3\cdot2^{k}t}\big)f\big\|_{L_q(\O;X)}^q\,\frac{dt}{t}\right)^{1/q}\\
  &= \sum_{k=-1}^\82^{-k}\left(\int_0^\8 \big\|t\partial T_t\big(T_{t}-T_{3t}\big)f\big\|_{L_q(\O;X)}^q\,\frac{dt}{t}\right)^{1/q}\\
  &=4\left(\int_0^\8\big\|t\partial T_t\big(T_{t}- T_{3t}\big)f\big\|_{L_q(\O;X)}^q\,\frac{dt}{t}\right)^{1/q}\,.
  \end{align*}

We are now in a position of using Lemma~\ref{analyticity0} with $p=q$.  Indeed, since $X$ is of martingale cotype $q$, so is $Y=L_q(\O; X)$. Then by \cite{pis1}, $Y$ can be renormalized into a uniformly convex space of power type $q$, that is, $Y$ admits an equivalent norm satisfying \eqref{conx}. Thus we have \eqref{pisier}; moreover, the constant $C$ there depends only on $q$ and the martingale cotype $q$ constant of $X$.

Therefore,
 $$\big\|t\partial T_t\big(T_{t}- T_{3t}\big)f\big\|_{L_q(\O;X)}\les \big\|\big(T_{t}- T_{3t}\big)f\big\|_{L_q(\O;X)}\,,\quad\forall\,t>0.$$
Combining the above inequalities together with Lemma~\ref{HN-lem}, we deduce
 $$\left(\int_0^\8\big\|t\partial T_tf\big\|_{L_q(\O;X)}^q\,\frac{dt}{t}\right)^{1/q}\les
 \left(\int_0^\8\big\|\big(T_{t}- T_{3t}\big)f\big\|_{L_q(\O;X)}^q\,\frac{dt}{t}\right)^{1/q}\les \big\|f\big\|_{L_q(\O;X)}\,.$$
This is \eqref{pqk} for $k=1$. To handle a general $k$, by the semigroup identity $T_{t+s}=T_tT_s$ once more, we have
 $$t^k\partial^k T_t=k^k\left(\frac{t}k\,\partial T_{\frac{t}k}\right)^k\,.$$
Thus, by \eqref{pisier} and the already proved inequality, we obtain
  \begin{align*}
  \int_0^\8\big\|t^k\partial^k T_tf\big\|_{L_q(\O;X)}^q\,\frac{dt}{t}
 &=k^k \int_0^\8\big\|\left(t\,\partial T_{t}\right)^kf\big\|_{L_q(\O;X)}^q\,\frac{dt}{t}\\
 &\les \int_0^\8\big\|t \partial T_{t}\,f\big\|_{L_q(\O;X)}^q\,\frac{dt}{t}\les  \big\|f\big\|_{L_q(\O;X)}^q\,.
  \end{align*}
The lemma is thus proved.
 \end{proof}

To show Theorem~\ref{heat} for any $1<p<\8$, we will use Stein's complex interpolation machinery. To that end, we will need the fractional integrals. For a (nice) function $\f$ on $(0,\,\8)$ define
 $$\mathrm I^\a\f(t)=\frac1{\Ga(\a)}\,\int_0^t (t-s)^{\a-1}\f(s)ds,\quad t>0.$$
The integral in the right hand side is well defined for any $\a\in\com$ with ${\rm Re}\,\a>0$; moreover, $\mathrm I^\a\f$ is analytic in the right  half complex plane ${\rm Re}\,\a>0$. Using integration by parts, Stein showed in \cite[section~III.3]{stein} that $\mathrm I^\a\f$ has an analytic continuation to the whole complex plane, which satisfies the following properties
 \begin{enumerate}[$\bullet$]
 \item $\mathrm I^\a\mathrm I^\b\f=\mathrm I^{\a+\b}\f$ for any $\a, \b\in\com$;
 \item $\mathrm I^0\f=\f$;
 \item $\mathrm I^{-k}=\partial^k\f$ for any positive integer $k$.
 \end{enumerate}
We will apply $\mathrm I^\a$ to $\f$ defined by $\f(s)=T_sf$ for a given function $f$ in $L_p(\O; X)$ and set
 $$\mathrm M^\a_tf=t^{-\a} \mathrm I^\a \f(t)\;\text{ with }\; \f(s)=T_sf.$$
Note that
 $$\mathrm M^1_tf=\frac1t\,\int_0^tT_sfds,\;\; \mathrm M^0_tf=T_tf\;\text{ and }\; \mathrm M^{-k}_tf=t^k \partial^k T_tf\;\text{ for }\;k\in\nat.$$

The following lemma is \cite[Theorem~2.3]{MTX}.

\begin{lem}\label{average}
 Let $q$ and $X$ be as in Theorem~\ref{heat}. Then for any $1<p<\8$ we have
  $$\left\|\left(\int_0^\8\big\|t\partial \mathrm M^1_tf\big\|_{X}^q\,\frac{dt}{t}\right)^{1/q}\right\|_{L_p(\O)}\les\big\|f\big\|_{L_p(\O;X)}\,,\quad\forall\,f\in L_p(\O; X).$$
 \end{lem}

 \begin{lem}\label{a-b}
 Let $\a$ and $\b$ be complex numbers such that ${\rm Re}\,\a>{\rm Re}\,\b>-1$. Then for any positive integer $k$
  $$\left(\int_0^\8\big\|t^k\partial^k \mathrm M^\a_tf\big\|_{X}^q\,\frac{dt}{t}\right)^{1/q}
  \le C e^{\pi |{\rm Im}(\a-\b)|}\left(\int_0^\8\big\|t^k\partial^k \mathrm M^\b_tf\big\|_{X}^q\,\frac{dt}{t}\right)^{1/q}\;\text{ on }\O,$$
 where $C$ is a constant depending only on  ${\rm Re}\,\a$ and ${\rm Re}\,\b$.
 \end{lem}

\begin{proof}
Using
  $\mathrm I^\a=\mathrm I^{\a-\b}\,\mathrm I^\b$,
 we write
 $$ \mathrm M^\a_t f
  =\frac{t^{-\a}}{\Ga(\a-\b)}\,\int_0^t(t-s)^{\a-\b-1}s^{\b} \mathrm M^\b_sf\,ds
  =\frac{1}{\Ga(\a-\b)}\,\int_0^1(1-s)^{\a-\b-1}s^{\b} \mathrm M^\b_{ts}f\,ds.
  $$
 Thus
 $$
 t^k\partial^k \mathrm M^\a_t f
 =\frac{1}{\Ga(\a-\b)}\,\int_0^1(1-s)^{\a-\b-1}s^{\b} (ts)^k\partial^k \mathrm M^\b_{ts}f\,ds,
 $$
which implies
  \begin{align*}
  \left(\int_0^\8\big\|t^k\partial^k \mathrm M^\a_tf\big\|_{X}^q\,\frac{dt}{t}\right)^{1/q}
  &\le \frac{1}{|\Ga(\a-\b)|}\,\int_0^1(1-s)^{{\rm Re}(\a-\b)-1}s^{{\rm Re}\,\b} \,ds
  \left(\int_0^\8\big\|t^k\partial^k \mathrm M^\b_{t}f\big\|_{X}^q\,\frac{dt}{t}\right)^{1/q}\\
  &\les\frac{1}{|\Ga(\a-\b)|}\,\left(\int_0^\8\big\|t^k\partial^k \mathrm M^\b_{t}f\big\|_{X}^q\,\frac{dt}{t}\right)^{1/q}\,.
  \end{align*}
Then the desired inequality follows from the following well known estimate on the $\Ga$-function:
 $$\forall\, x,y\in\real,\quad |\Ga(x+{\rm i}y)|\sim e^{-\frac\pi2\,|y|}|y|^{x-\frac12}\;\text{ as }\; y\to\pm\8$$
 (see \cite[p.~151]{tit}).
   \end{proof}

Combining Lemma~\ref{average} and Lemma~\ref{a-b} with $k=\b=1$, we get

  \begin{lem}\label{a>1}
 For any $1<p<\8$ and $\a\in\com$ with ${\rm Re}\,\a>1$
  $$\left\|\left(\int_0^\8\big\|t\partial \mathrm M^\a_tf\big\|_{X}^q\,\frac{dt}{t}\right)^{1/q}\right\|_{L_p(\O)}
  \le C e^{\pi |{\rm Im}\,\a|}\big\|f\big\|_{L_p(\O;X)}\,,\quad\forall\,f\in L_p(\O; X),$$
 where $C$ depends on ${\rm Re}\,\a$, $p$ and the martingale cotype $q$ constant of $X$.
 \end{lem}

 \begin{lem}\label{p=q2}
   For any $\a\in\com$
   \beq\label{q-a}
  \left\|\left(\int_0^\8\big\|t\partial \mathrm M^\a_tf\big\|_{X}^q\,\frac{dt}{t}\right)^{1/q}\right\|_{L_q(\O)}
  \le C e^{\pi |{\rm Im}\,\a|}\big\|f\big\|_{L_q(\O;X)}\,,\quad\forall\,f\in L_p(\O; X),
 \eeq
 where $C$ depends on ${\rm Re}\,\a$ and the martingale cotype $q$ constant of $X$.
  \end{lem}

 \begin{proof}
    Combining Lemma~\ref{p=q} and Lemma~\ref{a-b} with $\b=0$, we deduce that for a positive integer $k$ and $\a\in\com$ with ${\rm Re}\,\a>0$
  \beq\label{iterate}
  \left\|\left(\int_0^\8\big\|t^k\partial^k \mathrm M^\a_tf\big\|_{X}^q\,\frac{dt}{t}\right)^{1/q}\right\|_{L_q(\O)}
  \le C e^{\pi |{\rm Im}\,\a|}\big\|f\big\|_{L_q(\O;X)}\,,\quad\forall\,f\in L_q(\O; X),
  \eeq
where $C$ depends on $k$, ${\rm Re}\,\a$ and the martingale cotype $q$ constant of $X$. In particular, when $k=1$, we get \eqref{q-a} for any $\a$ such that ${\rm Re}\,\a>0$.

To deal with the general case, we will use an iteration procedure. Noting that for any  $\a\in\com$
 $$\partial  \mathrm M^\a_t=-\a t^{-1} \mathrm M^\a_t+ t^{-1} \mathrm M^{\a-1}_t,$$
we have
 \beq\label{a-1}
  t^k\partial^k \mathrm M^{\a-1}_t
  =(k+\a) t^k\partial^k \mathrm M^{\a}_t +  t^{k+1}\partial^{k+1} \mathrm M^{\a}_t\,.
  \eeq
This shows that if \eqref{iterate} holds for $\mathrm M^{\a}$, so does it for $\mathrm M^{\a-1}$ instead of $\mathrm M^{\a}$ (with a different constant). Therefore, by  what already proved,  we deduce that \eqref{iterate} holds for any  $\a\in\com$ with ${\rm Re}\,\a>-1$. Repeating this argument, we obtain \eqref{iterate} for any  $\a\in\com$. In particular for $k=1$, we have \eqref{q-a}.
  \end{proof}

 Now we are ready to show Theorem~\ref{heat} (i).

 \begin{proof}[Proof of Theorem~\ref{heat} (i)]
  We will prove the following more general statement: Under the assumption of assertion (i), we have for any $\a\in\com$
  \beq\label{heat a}
  \left\|\left(\int_0^\8\big\|t^k\partial^k \mathrm M^\a_tf\big\|_{X}^q\,\frac{dt}{t}\right)^{1/q}\right\|_{L_p(\O)}
  \les\big\|f\big\|_{L_p(\O;X)}\,,\quad\forall\,f\in L_p(\O; X).
  \eeq
Assertion (i) corresponds to \eqref{heat a} for $\a=0$.

Fix $\a\in \com$. Choose $\theta\in(0,1),\; r\in(1,\8)$, $\a_0, \a_1\in\com$
such that
 $$\frac{1}{p}=\frac{1-\theta}{q} + \frac{\theta}{r}\,,
 \quad
 \a=(1-\theta)\a_0 +\theta\,\a_1,\quad{\rm Re}\,\a_1>1\;\text{ and }\; {\rm Im}\,\a_0={\rm Im}\,\a_1={\rm Im}\,\a.$$
Then by the classical complex interpolation on vector-valued $L_p$-spaces (cf. \cite{bl}), we have
 $$L_p(\O; X)=\big(L_q(\O; X),\, L_r(\O; X)\big)_\theta\,.$$
Thus for any  $f\in L_p(\O;X)$ with norm less than $1$ there exists a continuous function $F$ from the closed strip $\{z\in\com: 0\le {\rm Re}\, z\le1\}$ to $L_q(\O; X)+L_r(\O; X)$, which is analytic in the interior and satisfies
  $$F(\theta)=f,\quad \sup_{y\in\real}\big\|F({\rm i}y)\big\|_{L_q(\O; X)}<1\;\text{ and }\; \sup_{y\in\real}\big\|F(1+{\rm i}y)\big\|_{L_r(\O; X)}<1.$$
 Define
  $$\mathcal F_t(z)=e^{z^2-\theta^2}\,t\partial \mathrm M^{(1-z)\a_0+z\a_1}_tF(z).$$
Viewed as a function of $z$ on the strip $\{z\in\com: 0\le {\rm Re}\, z\le1\}$, $\mathcal F$ takes values in $L_q(\O; L_q(\real_+; X))+L_r(\O; L_q(\real_+; X))$, where $\real_+$ is equipped with the measure $\frac{dt}t$. By the analyticity of $\mathrm M^{(1-z)\a_0+z\a_1}$ in $z$, we see that $\mathcal F$ is analytic in the interior of the strip. Moreover,
by Lemma~\ref{p=q2}
 $$\left\|\left(\int_0^\8\big\|\mathcal F_t({\rm i}y)\big\|_{X}^q\,\frac{dt}{t}\right)^{1/q}\right\|_{L_q(\O)}
  \le C_0' e^{-y^2-\theta^2}\,e^{\pi (|{\rm Im}\,\a|+|{\rm Re}(\a_1-\a_0)y|)}\,,\quad\forall\,y\in\real,$$
where $C_0'$ is a constant depending on $\a, \a_0,\a_1$ and $X$. Hence
 $$\sup_{y\in\real} \left\|\mathcal F({\rm i}y)\right\|_{L_q(\O; L_q((\real_+,\frac{dt}t); X))}\le C_0\,.$$
Similarly, Lemma~\ref{a>1} implies
 $$\sup_{y\in\real} \left\|\mathcal F(1+{\rm i}y)\right\|_{L_r(\O; L_q((\real_+,\frac{dt}t); X))}\le C_1\,.$$
We then deduce that $\mathcal F(\theta)$ belongs to the complex interpolation space
 $$ \big(L_q(\O; L_q(\real_+; X)),\, L_r(\O; L_q(\real_+; X))\big)_{\theta}$$
with norm majorized by $C_0^{1-\theta}C_1^\theta$. However, the latter space coincides with  $L_p(\O; L_q(\real_+; X))$ isometrically. Since
 $$\mathcal F_t(\theta)=t\partial \mathrm M^{\a}_tF(\theta)=t\partial \mathrm M^{\a}_tf,$$
we get \eqref{heat a} for $k=1$. Then using \eqref{a-1} and an induction argument, we derive \eqref{heat a} for any $k$. Thus the theorem is completely proved.
 \end{proof}


\section{Proofs of Theorem~\ref{analyticity-discrete} and Theorem~\ref{discrete}}
\label{pf-analyticity-discrete}


The main part of  Theorem~\ref{analyticity-discrete} is already contained in Lemma~\ref{spectral estimate}. Armed with that lemma, we can easily show Theorem~\ref{analyticity-discrete}. Let us first  recall the following well known characterization of the analyticity of power bounded operators (cf. \cite[Theorem~2.3]{b} and \cite[Theorem~4.5.4]{N}). Let $\mathbb{D}$  denote the open unit disc of the complex plane and $\T$ the boundary of $\mathbb{D}$.

\begin{lem}\label{analyticity}
 Let $T$ be a power bounded linear operator on a Banach space $Y$. Then $T$ is analytic iff the semigroup $\{e^{t(T-1)}\}_{t>0}$ is analytic and $\s(T)\subset\mathbb{D}\cup\{1\}$.
\end{lem}

 \begin{proof}[Proof of Theorem~\ref{analyticity-discrete}]
  Note that $\{e^{t(T-1)}\}_{t>0}$ is a symmetric diffusion semigroup on $(\O,\A,\mu)$. Thus, by Lemma~\ref{analyticity0}, its extension to $Y=L_p(\O; X)$ is analytic. Then Theorem~\ref{analyticity-discrete} immediately follows from Lemmas~\ref{spectral estimate} and \ref{analyticity}.
     \end{proof}

The difficult part ( Lemma~\ref{spectral estimate}) of the above proof concerns the quantitative dependence on the geometry of $X$  of the angle $\g$ of the Stolz  domain which contains the spectrum of the operator $T$. If we only need to show the analyticity of $T$ on $Y$, the proof can be largely shortened by virtue of the following simple fact which, together with  Lemma~\ref{Rota}, ensures that $\s(T)\subset\mathbb{D}\cup\{1\}$.

\begin{rk}
 Let $P$ be a contractive linear projection on a uniformly convex Banach space $Y$. Then $\|\l-P\|< 2$ for any $\l\in\T\setminus\{-1\}$.
\end{rk}

This remark is a weaker form of Lemma~\ref{spectral estimate}. Let $\l\in\T$ such that $\|\l-P\|=2$. Choose a sequence $\{y_k\}$ of unit vectors in $Y$ such that $\|y_k-Py_k\|\to2$ as $k\to\8$. Then the uniform convexity of $Y$ implies  $\|\l y_k+Py_k\|\to0$. However,
 $$|\l+1|\,\|Py_k\|= \|P(\l+P)y_k\| \le \|(\l+P)y_k\| \; \text{ and }\; \|Py_k\|\ge \|\l y_k -Py_k\|-1\to  1.$$
It thus follows that $|\l+1|= 0$, that is, $\l = -1$.

\smallskip

Now we turn to the proof of Theorem~\ref{discrete}. We first deduce assertion  (ii)  from assertion  (i) by duality as in the continuous case. Under the assumption of Theorem~\ref{discrete} and Pisier's renorming theorem \cite{pis1}, we can assume that $X$ is uniformly convex.

\begin{proof}[Proof of Theorem~\ref{discrete} {\rm(ii)}]  Using the spectral resolution of the identity of $T$ on $L_2(\O)$, we obtain
   $$\big\|f-\mathsf F(f)\big\|^2_{L_2(\O)}=\sum_{n=1}^\8n\big\|T^{n-1}(1-T^2)f\big\|_{L_2(\O)}^2\,,\quad f\in L_2(\O).$$
Polarizing this identity, we deduce, for $f\in L_1(\O)\cap L_\8(\O)\otimes X$ and $g\in L_1(\O)\cap L_\8(\O)\otimes X^*$, that
   \begin{align*}
   \left|\int_{\O}\la f-\mathsf F(f),g-\mathsf F(g)\ra d\mu\right|
   \le &\left\|\left(\sum_{n=1}^\8 n^{q'-1}\big\|T^{n-1}(1-T^2)g\big\|_{X^*}^{q'}\right)^{1/q'}\right\|_{L_{p'}(\O)}\\
      \cdot\,&\left\|\left(\sum_{n=1}^\8 n^{q-1}\big\|T^{n-1}(1-T^2)f\big\|_{X}^{q}\right)^{1/q}\right\|_{L_{p}(\O)}\\
      \le 4&\left\|\left(\sum_{n=1}^\8 n^{q'-1}\big\|T^{n-1}(1-T)g\big\|_{X^*}^{q'}\right)^{1/q'}\right\|_{L_{p'}(\O)}\\
      \cdot\,&\left\|\left(\sum_{n=1}^\8 n^{q-1}\big\|T^{n-1}(1-T)f\big\|_{X}^{q}\right)^{1/q}\right\|_{L_{p}(\O)}\,.
    \end{align*}
Thus under the assumption of (ii) and admitting (i), we obtain
 $$\big\|f-\mathsf F(f)\big\|_{L_{p}(\O; X)}\les \left\|\left(\sum_{n=1}^\8 n^{q-1}\big\|T^{n-1}(1-T)f\big\|_{X}^{q}\right)^{1/q}\right\|_{L_{p}(\O)}\,.$$
Thus assertion (ii) is proved.
  \end{proof}

We will need some preparations on the $H^\8$ functional calculus for the proof of Theorem~\ref{discrete} (i). Our reference for the latter subject is \cite{CDMY}.
Let $A$ be a sectorial operator on a Banach space $Y$ with angle $\g$ and $\o>\g$. Define $H_0^\8(\Sigma_\o) $ to be the space of  all bounded analytic functions $\f$ on the sector $\Sigma_\o$ for which there exist two positive constants $s$ and $C$ such that
 $$|\f(z)| \leq C\min\{|z|^s,\, |z|^{-s}\},\quad \forall z\in \Sigma_\o.$$
For any $\f\in H^\8_0(\Sigma_\o)$, we define
 $$\f(A)=\frac{1}{2\pi {\rm i}}\,\int_{\Gamma_\theta} \f(z)(z-A)^{-1}\, dz,$$
where $\theta\in(\g,\,\o)$ and $\Gamma_\theta$ is the
boundary $\partial\Sigma_\theta$ oriented counterclockwise. Then $\f(A)$ is a bounded operator on $Y$.

The following result is a variant of \cite[Theorem~5]{MY}. The proof there works equally for the present setting without change. This was pointed to us by Christian Le Merdy (see \cite [page~719]{LM}).

\begin{lem}\label{Mac}
 Let  $1<q<\8$ and $\f, \psi\in H^\8_0(\Sigma_\o)$ with
 $$\int_0^\8\psi(t)\,\frac{dt}t\neq0.$$
Then there exists a positive constant $C$, depending only on $\f, \psi$ and $q$, such that
 $$\left(\int_0^\8\big\|\f(tA)y\big\|^q\,\frac{dt}t\right)^{1/q}\le C\left(\int_0^\8\big\|\psi(tA)y\big\|^q\,\frac{dt}t\right)^{1/q}\,,\quad \forall y\in Y .$$
 \end{lem}

\smallskip

\begin{proof}[Proof of Theorem~\ref{discrete} {\rm(i)}]
  We will follow the pattern set up in the proof of Theorem~\ref{heat}. The main difficulty is to prove the following discrete analogue of Lemma~\ref{p=q}:
 \beq\label{p=q: discrete}
 \sum_{n=1}^\8n^{q-1}\big\|T^n(T-1)f\big\|_{L_q(\O; X)}^q\les \big\|f\big\|_{L_q(\O; X)}^q\,,\quad \forall f\in L_q(\O; X).
 \eeq
Contrary to Lemma~\ref{p=q}, the proof of the above inequality is much more involved.  We will adapt the proof of \cite[Proposition~3.2]{LMX} which is based on the $H^\8$ functional calculus.

By Theorem~\ref{analyticity-discrete}, $T$ is analytic as an operator on $Y=L_q(\O; X)$ and we have \eqref{resolvent}.
Let $A=1-T$. Then $A$ is a sectorial operator on $Y$ with angle $\g$. Fix $\theta\in(\g,\, \pi/2)$. Let $L_\theta$ be the boundary of $1-B_\theta$ oriented counterclockwise (see figure~2).

\begin{figure}[ht]
\vspace*{2ex}
\begin{center}
\includegraphics[scale=0.45]{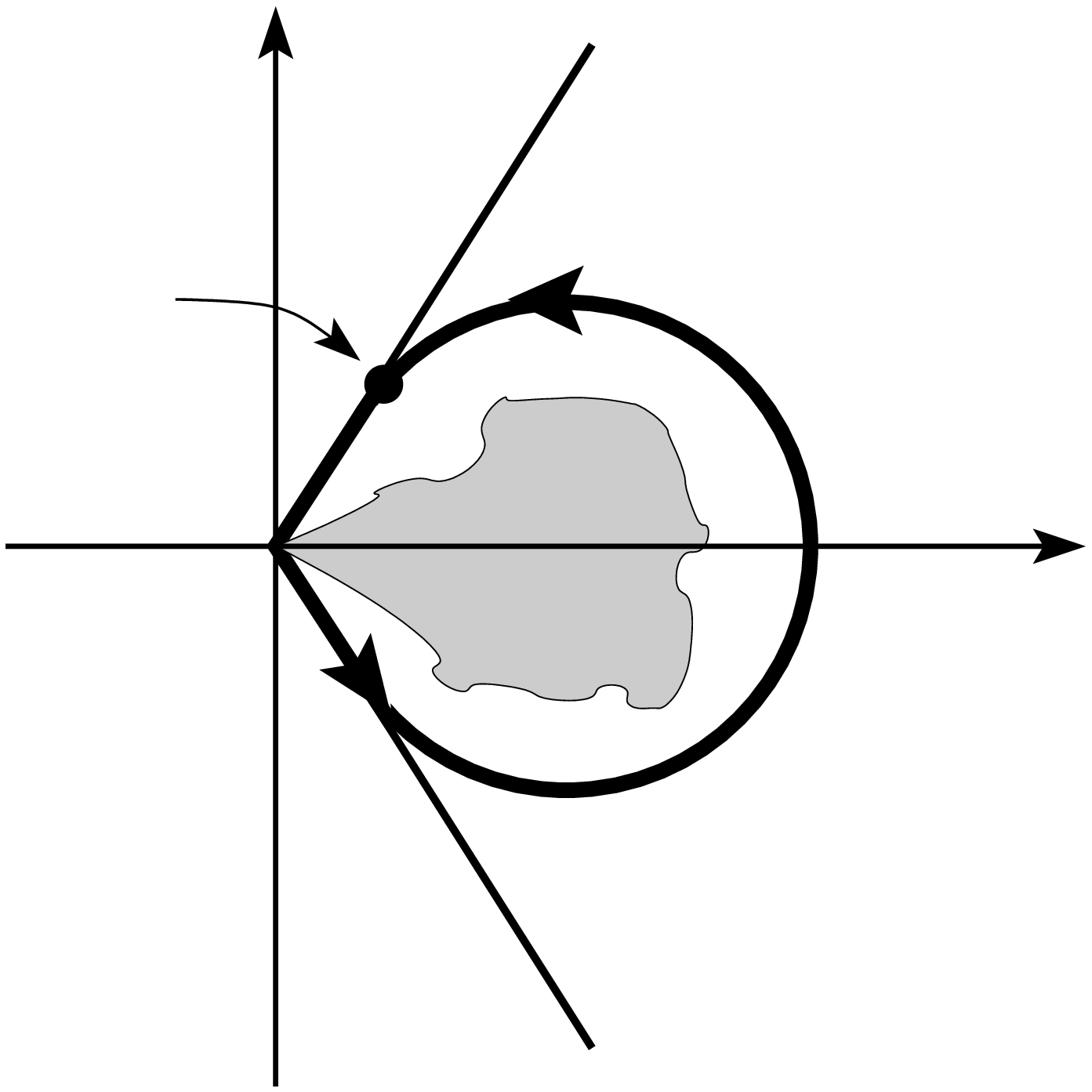}
\begin{picture}(0,0)
\put(-148,79){{\footnotesize $0$}}
\put(-97,174){{\small $\theta$}}
\put(-75,132){{\small $L_\theta$}}
\put(-100,98){{\small $\sigma(A)$}}
\put(-201,131){{\small $\cos(\theta)e^{{\rm i}\theta}$}}
\end{picture}
\end{center}
\caption{ \label{f2}}
\end{figure}

Let $\f_n(z)=n^{1/q'} z (1-z)^n$. Then by  the Dunford functional calculus
$$
\frac{1}{2\pi {\rm i}}\,\int_{L_{\theta}} \f_n(z)(z -A)^{-1}\, dz=\f_n(A)\;\text{ and }\;
\,\frac{1}{2\pi{\rm i}}\,\int_{L_{\theta}} \f_n(z)(z +A)^{-1}\, dz=0.
$$
Thus
$$
n^{1/q'} T^n(1-T)=\f_n(A) = \,\frac{1}{\pi{\rm i}}\,\int_{L_{\theta}} \f_n(z) A(z-A)^{-1}(z +A)^{-1}\, dz\,.
$$
Fix $f$  in the unit ball of  $Y$. Then
 $$
  \sum_{n=1}^\8n^{q-1}\big\|T^n(T-1)f\big\|_Y^q
 \les \int_{L_{\theta}} \sum_{n=1}^\8|\f_n(z)|^q\,\big\|A(z-A)^{-1}(z +A)^{-1}f\big\|_Y^q\, |dz|.
 $$
Note that for any $z\in L_{\theta}$, an elementary calculation shows that
 $$\sum_{n=1}^\8|\f_n(z)|^q
 \le \sup_{\l\in B_\theta}\,\sum_{n=1}^\8 n^{q-1}|\l|^{nq}|1-\l|^q
 \les \sup_{\l\in B_\theta}\,\frac{|1-\l|^q}{(1-|\l|)^q}\les 1,$$
where the relevant constants depend only on $q$ and $\theta$. On the other hand, by the $H^\8$ functional calculus, $A^{1/q}(z +A)^{-1}$ is a bounded operator on $Y$. Then we deduce
 $$\sum_{n=1}^\8n^{q-1}\big\|T^n(T-1)f\big\|_Y^q
 \les \int_{L_{\theta}} \big\|A^{1/q'}(z-A)^{-1}f\big\|_Y^q\, |dz|.$$
The contour $L_{\theta}$ is the juxtaposition of a part $L_{\theta,1}$ of $\Gamma_{\theta}$ (recalling that $\Gamma_\theta$ is the
boundary of the sector $\Sigma_\theta$)
and the curve $L_{\theta,2}$ going from  $\cos(\theta) e^{-{\rm i}\theta}$ to $\cos(\theta) e^{{\rm i}\theta}$
counterclockwise along the circle of center $1$ and radius $\sin\theta$. Accordingly,
 $$\int_{L_{\theta}} \big\|A^{1/q'}(z-A)^{-1}f\big\|_Y^q\, |dz|=\int_{L_{\theta, 1}}\big\|A^{1/q'}(z-A)^{-1}f\big\|_Y^q\, |dz|
 +\int_{L_{\theta,2}} \big\|A^{1/q'}(z-A)^{-1}f\big\|_Y^q\, |dz|.$$
Since $L_{\theta,2}\cap \sigma(A)=\emptyset$, the function $z\mapsto\|A^{1/q'}(z-A)^{-1}\|$ is bounded on $L_{\theta,2}$. Thus the second integral in the right hand side above is majorized by a constant independent of $f$ (recalling that $\|f\|_Y\le 1$).  For the first one, we have
  \begin{align*}
  \int_{L_{\theta, 1}}\big\|A^{1/q'}(z-A)^{-1}f\big\|_Y^q\, |dz|
 &\le \sum_{\e=\pm1}\, \int_0^\8 \big\|A^{1/q'}(te^{\e{\rm i}\theta}-A)^{-1}f\big\|_Y^q\, dt \\
 &=\sum_{\e=\pm1}\,\int_0^\8 \big\|(tA)^{1/q'}(e^{\e{\rm i}\theta}-tA)^{-1}f\big\|_Y^q\, \frac{dt}t \\
 &=\sum_{\e=\pm1}\,\int_0^\8 \big\|\f_\e(tA)f\big\|_Y^q\, \frac{dt}t \,,
  \end{align*}
where
$$
\f_\e(z)=\frac{z^{1/q'}}{e^{\e{\rm i}\theta} -z} \,,\quad \e=\pm1\,.
$$
Note that $\f_\e\in H^\8_0(\Sigma_\o)$ for  $\o\in(\theta,\, \pi/2)$. On the other hand, the function $\psi$ defined by $\psi(z)=ze^{-z}$ belongs to $H^\8_0(\Sigma_\o)$ too.   Thus applying Lemma~\ref{Mac}, we get
 $$\int_0^\8 \big\|\f_\e(tA)f\big\|_Y^q\, \frac{dt}t\les \int_0^\8 \big\|\psi(tA)f\big\|_Y^q\, \frac{dt}t
 =\int_0^\8 \big\|t\,\partial T_tf\big\|_Y^q\, \frac{dt}t\,,$$
where  $\{T_t\}_{t>0}=\{e^{-tA}\}_{t>0}$ is the semigroup already used at the beginning of the proof of Theorem~\ref{analyticity-discrete}. Thus by Lemma~\ref{p=q},
 $$ \int_0^\8 \big\|\psi(tA)f\big\|_Y^q\, \frac{dt}t\les \big\|f\big\|_Y\les 1.$$
Combining all preceding  inequalities, we finally get
 $$\sum_{n=1}^\8n^{q-1}\big\|T^n(T-1)f\big\|_{L_q(\O; X)}^q\les 1$$
 for any $f$ in the unit ball of $Y$. This yields \eqref{p=q: discrete} by homogeneity.

 \smallskip

Armed with \eqref{p=q: discrete}, we can finish the proof of Theorem~\ref{discrete} (i)  by Stein's complex interpolation machinery as in the continuous case. To that end, first recall that   Lemma~\ref{average} is deduced  by approximation from its discrete analogue in \cite{MTX}. Thus, although not explicitly stated there,  the discrete analogue of Lemma~\ref{average} is indeed  obtained during the proof of  \cite[Theorem~2.3]{MTX}. Then the interpolation arguments in the previous section can be modified to the present discrete setting. We refer the reader to \cite{stein-m} for the necessary ingredients. However, note that the presentation of  \cite{stein-m} is quite brief, it is developed in detail in \cite{JX}. We leave the details to the reader. Thus the proof of Theorem~\ref{discrete} is complete.
 \end{proof}


\section{Open problems}
\label{pb}


We conclude this article by some open problems. The first one concerns Theorem~\ref{discrete}. Note that in that theorem the contraction $T$ is assumed to be the square of another symmetric Markovian operator. Compared with the continuous case, this assumption is natural since every operator in a symmetric diffusion semigroup is automatically  the square of a symmetric Markovian operator.  However, a less restrictive assumption would be that $T$ is a selfadjoint contraction on  $L_2(\O)$ and its spectrum does not contain $-1$. Under this assumption, $T$ is analytic. If in addition $T$ is a  contraction on $L_p(\O)$ for every $1\le p\le\8$, then $T$ is also analytic on $L_p(\O)$ for every $1<p<\8$.

\begin{pb} \label{self}
 Let $T$ be a positive contraction on $L_p(\O)$ for every $1\le p\le\8$ with $T1=1$. Assume that $T$ is selfadjoint on  $L_2(\O)$ and its spectrum does not contain $-1$.
 \begin{enumerate}[{\rm(i)}]
 \item Let $X$ be a uniformly convex Banach space. Is the extension of $T$ to $L_p(\O; X)$ analytic for every $1<p<\8$ $($or equivalently, for one $1<p<\8)$?
 \item Let $X$ be a Banach space of martingale cotype $q$ and $1<p<\8$. Does one have
  $$\left\|\left(\sum_{n=1}^\8 n^{q-1}\big\|(T^n-T^{n-1})f\big\|_X^q\right)^{1/q}\right\|_{L_p(\O)}\les  \big\|f\big\|_{L_p(\O; X)}\,,\quad\forall\, f\in L_p(\O; X)?$$
  \end{enumerate}
\end{pb}

An affirmative answer to part (i) would imply the same for part (ii). In the spirit of \cite{pis2}, one can ask a similar question as part (i) for K-convex  $X$. In fact, we do not know whether Theorem~\ref{analyticity-discrete} holds for K-convex targets (see \cite{LL}  for related results). This is the discrete analogue of Problem~11 (i) of \cite{X} for symmetric diffusion semigroups.

\begin{pb}\label{target}
 Does Theorem~\ref{analyticity-discrete} remain true if $X$ is assumed K-convex?
\end{pb}

\begin{rk}
The answers to Problem~\ref{self} (i) and Problem~\ref{target} are both positive if $X$ is a complex interpolation space between a Hilbert space and a Banach space.This is the case if $X$ is a K-convex Banach lattice thanks to \cite{pis4}. More generally, let $(X_0,\, X_1)$ be a compatible pair of Banach spaces, and let $X=(X_0,\, X_1)_\theta$ with $0<\theta<1$. Assume that $T$ is a contraction on both $X_0$ and $X_1$, and $T$ is analytic on $X_1$. Then $T$ is analytic on $X$ too.
\end{rk}

Indeed, since the semigroup $\{e^{(T-1)t}\}_{t>0}$ is analytic on $X_1$, by Stein's complex interpolation, it is analytic on $X$ too. Thus by Lemma~\ref{analyticity}, it remains to show that as an operator on $X$, the spectrum of $T$ intersects $\T$ at most at the point $1$. The latter is equivalent to $\lim_{n\to\8}\|T^n(T-1): X\to X\|=0$, thanks to  Katznelson and Tzafriri's theorem \cite{KT}. Using the analyticity of $T$ on $X_1$ and interpolation, we get
 $$\|T^n(T-1): X\to X\|\les \frac1{n^\theta}\,.$$
 So we are done.

\smallskip

Hyt\"onen \cite{Hy} studied another variant of Stein's inequality \eqref{LPS} in the vector-valued setting. Like \cite{MTX}, his main theorem deals with the Poisson semigroup subordinated to a symmetric diffusion semigroup for a general UMD space $X$, except when $X$ is a complex interpolation space between  a Hilbert space and another UMD space. In the same spirit of this article, one may ask whether the main result of \cite{Hy} remains true for any symmetric diffusion semigroup and any UMD space $X$. It is easier to formulate this problem in the discrete case as follows. Let $T$ be as in Theorem~\ref{discrete}.

\begin{pb}
 Let $T$ be as in Theorem~\ref{discrete}, $X$ be a UMD space and $1<p<\8$. Does one have
  $$\mathbb{E}\big\|\sum_{n=1}^\8 \e_n\sqrt{n}\,(T^n-T^{n-1})f\big\|_{L_p(\O; X)}\approx  \big\|f-\mathsf F(f)\big\|_{L_p(\O; X)}\,,\quad\forall\, f\in L_p(\O; X)?$$
  \end{pb}
Here $\{\e_n\}$ is a sequence of symmetric random variables taking values $\pm1$ on a probability space and $\mathbb{E}$ is the corresponding expectation.

\bigskip \n{\bf Acknowledgements.} This work is partially supported by NSFC grants No. 11301401 and the French project ISITE-BFC (ANR-15-IDEX-03) and IUF.

\bigskip


\end{document}